\documentclass[12pt, english]{article}

\usepackage{babel}
\usepackage[latin1]{inputenc}
\usepackage{amsmath}
\usepackage{amssymb}
\usepackage{amsthm}
\usepackage{IEEEtrantools}
\usepackage{skak}
\usepackage{xskak}
\usepackage{chessfss}
\usepackage{xkeyval}
\usepackage{xifthen}
\usepackage{pgfcore}
\usepackage{pgfbaseshapes}
\usepackage{pst-node}

\usepackage{pgf}
\usepackage{tikz}
\usetikzlibrary{arrows}
\usetikzlibrary{decorations.pathmorphing}
\usetikzlibrary{backgrounds}
\usetikzlibrary{fit}
\usepackage{url}

\usepackage{chessboard}

\theoremstyle{plain}
\newtheorem{Thm}{Theorem}[section]
\newtheorem{Prop}[Thm]{Proposition}

\theoremstyle{definition}

\newtheorem{Def*}{Definition}

\DeclareSymbolFont{extraup}{U}{zavm}{m}{n}
\DeclareMathSymbol{\varheart}{\mathalpha}{extraup}{86}
\DeclareMathSymbol{\vardiamond}{\mathalpha}{extraup}{87}

\begin{document}

\title{Endgames in bidding chess}

\author{Urban Larsson\footnote{Urban Larsson, The Faculty of Industrial Engineering and Management,
Technion--Israel Institute of Technology,  Haifa 3200003, Israel, \texttt{urban031@gmail.com}} and Johan W\"astlund\footnote{Department of Mathematical Sciences, Chalmers University of Technology, SE-412 96 Gothenburg, Sweden, \texttt{wastlund@chalmers.se}}
}

\date{\small \today}  
\maketitle

\begin{abstract}
Bidding chess is a chess variant where instead of alternating play, players bid for the opportunity to move. Generalizing a known result on so-called Richman games, we show that for a natural class of games including bidding chess, each position can be assigned rational upper and lower values corresponding to the limit proportion of money that Black (say) needs in order to force a win and to avoid losing, respectively.  

We have computed these values for all three-piece endgames, and in all cases, the upper and lower values coincide. Already with three pieces, the game is quite complex, and the values have denominators of up to 138 digits. \end{abstract}


\section{Bidding Chess}
In chess, positions with only three pieces (the two kings and one more piece) are perfectly understood. The only such endgame requiring some finesse is king and pawn versus king, but even that endgame is played flawlessly by amateur players. 
In this article we investigate a chess variant where already positions with three pieces exhibit a complexity far beyond what can be embraced by a human.
 
\emph{Bidding Chess} is a chess variant where instead of the two players alternating turns, the move order is determined by a bidding process. Each player has a stack of chips and at every turn, the players bid for the right to make the next move. The highest bidding player then pays what they bid to the opponent, and makes a move. The goal is to capture the opponent's king, and therefore there are no concepts of check, checkmate or stalemate. 

As the total number of chips tends to infinity, there is in each position a limit proportion of chips that a player needs in order to force a win. We have computed these limits for all positions with three pieces, and the results (see for example Figure~\ref{F:complex}) show that already with such limited material, the game displays a remarkable intricacy. 

Similar \emph{bidding games} were introduced by David Richman in the 1980's, and presented in \cite{LLPU, LLPSU} only after his tragic death. Bidding chess has been discussed in \cite{Beasley, BP, DP}.

There are various reasonable protocols for making bids and handling situations of equal bids \cite{DP}. The bids can be \emph{secret}, meaning they are written down on slips of paper and then simultaneously revealed, or \emph{open}, where one player makes a bid and the other chooses between accepting (taking the money) or rejecting (paying the same amount and making a move). The open scheme is in theory equivalent to giving the choosing player a tiebreak advantage. 

For actual play we suggest an open scheme where the player who made the last move bids for the next. In the initial position, Black is considered to have made the last move (since it is normally White's turn), and starts the game by bidding for the first move. 

The bidding player can also make a \emph{negative} bid, meaning that the player who makes the next move will get paid for their trouble. As we shall see in Section~\ref{S:zugzwang}, there actually are \emph{zugzwang} positions calling for such bids. A player currently holding $n$ chips can therefore bid any integer in the range $-n,\dots, n$. A player who has more chips than their opponent can ensure the right to make the next move by bidding more than their opponent's bankroll, or force their opponent to move by similarly making a large negative bid.  

If the game is played with a small number of chips, the tiebreak scheme and the discreteness of the bidding options can affect the strategy \cite{DP}. However, as the number of chips grows, the game approaches a ``limit'' corresponding to a continuous model where one can bid any real amount. With continuous bidding we may assume that the total amount of money is 1. It turns out then that even a consistent tiebreak advantage is worth less than any positive amount of money. 


In the following, we therefore assume that the game is played with continuous money. We will ignore the bidding scheme and the tiebreak rules, since these will affect the outcome under optimal play only when the players' bankrolls are exactly at certain thresholds. Our discussion will focus on analyzing and computing these thresholds.


Figure~\ref{F:rookmate} shows a position that would be checkmate in ordinary chess.
\begin{figure} [h]
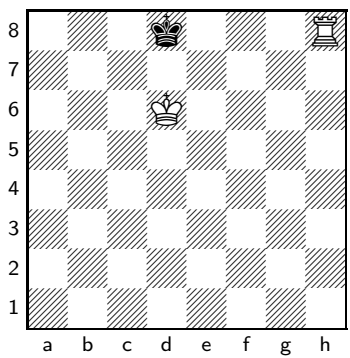

\storechessboardstyle{8x8}{maxfield=h8, smallboard, marginbottomwidth=0.5em}
\begin{center}
\chessboard[style=8x8, setwhite={Rh8,Kd6},addblack={Kd8},showmover=false]
\end{center}
\caption{A position with value $3/4$.}
\label{F:rookmate}
\end{figure}
White's rook is threatening the black king, and White will therefore go \emph{all in} (offering all their chips) before the next move. In order for Black to survive, they must win this bid, thereby doubling White's bankroll. Black's best option is now to move their king to one of the squares c7, d7 or e7, adjacent to the white king. At this point, whoever has more money will win the next bid and capture their opponent's king. 

The conclusion is that if White's bankroll is larger than $1/4$, they will be able to make one of the next two moves and win, while if it is smaller than $1/4$, Black is able to make two consecutive moves and capture the white king. At the threshold where White has exactly $1/4$ of the money, the outcome depends on the tiebreak scheme, but it still makes sense to say that this position has value $3/4$ for White. 


\section{Random turn games}
One of David Richman's insights was that there is a certain equivalence of bidding games to \emph{random turn} games \cite{LLPU, LLPSU, PR, PSSW07}. In a random turn game, the move order is determined by flipping a coin (just before each move, so you have to make your move before you know who plays next). The position in  Figure~\ref{F:rookmate} has value $3/4$ for White also in random turn chess. If White wins the next coin flip, the game is over, while if Black wins it, they will play Kd7 and the next coin flip decides the game.

The equivalence between bidding and random turn games can be understood inductively. If for the moment we disregard the possibility of draws (to which we shall return), we can write $\alpha(P)$ for the probability that White wins the random turn game from a given position $P$. Let $\alpha(P_w)$ be the probability of White winning from the position $P_w$ obtained after White's best move (that is, conditioning on White winning the next coin flip), and $\alpha(P_b)$ the probability of White winning after Black's best (from their perspective) move. Then provided the coin is fair, \begin{equation} \label{rec} \alpha(P) = \frac{\alpha(P_w) + \alpha(P_b)}2.\end{equation}

There is a bit of circularity in this argument, since the winning probabilities are what defines the ``best'' moves, so to make the argument rigorous we should consider the probability of white winning in at most $n$ moves, and then take the large $n$ limit (we will return to this issue). But the point is that equation \eqref{rec} has an interpretation also for the bidding game. We can think of the values $\alpha(P)$, $\alpha(P_w)$ and $\alpha(P_b)$ as the amounts of money that White can afford Black to have and still win the game. If we know how much money we will need after the next move, both if White and if Black makes that move, then the amount we need in the current position is the average of those two numbers, since we can then bid half their difference and win whether the opponent accepts or rejects. So \eqref{rec} holds also with that interpretation. By induction, it follows that the amount of money we can allow our opponent to have and still win is the same as our probability of winning the random turn game.

\section{Outline}
The main results of this study are as follows. Every partizan combinatorial game on a finite number of positions has rational upper and lower values (see Section~\ref{S:rational}). These values represent, under random turn play, the maximum probabilities that White can obtain of not losing, and of winning, respectively. Under bidding play, the same values represent the amount of money that Black needs in order to force a win and to avoid losing respectively. 

For all three-piece chess positions, the upper and lower values coincide. This is proved via computer calculation, see Section~\ref{S:verifying}, and we have no theoretical explanation for why this had to be the case. The calculation revealed that the game is extraordinarily complex. In particular there is a positions with king and knight versus king whose (common upper and lower) value has a 138-digit denominator. We also discuss some other positions of special interest, for example the existence of chess positions (with more than three pieces) with distinct upper and lower values (so-called \emph{nontrivial Richman intervals}), and of positions of \emph{zugzwang}, that is, positions requiring negative bids.


\section{Examples} \label{S:examples}
In some cases, values of positions in bidding chess can be conveniently calculated by instead analyzing random turn chess. Consider for instance a position with two bare kings: 

\begin{figure} [h]
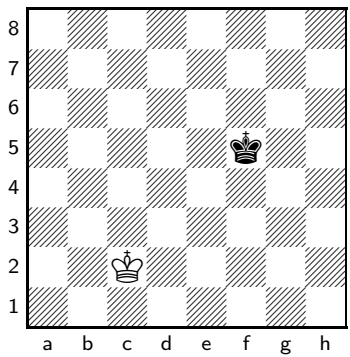

\storechessboardstyle{8x8}{maxfield=h8, smallboard, marginbottomwidth=0.5em}
\begin{center}
\chessboard[style=8x8, setwhite={Kc2},addblack={Kf5},showmover=false]
\end{center}
\caption{A position with value $1/2$.}
\label{F:barekings}
\end{figure}

In ordinary chess this position is a draw since no king can move to a square adjacent to the other. And if none of the players is willing to take a risk, the random turn game too will be drawn. But a player can guarantee a winning probability of $1/2$ even if the other player is satisfied with a draw. This is a simple consequence of the laws of probability: At some point you will get a run of six consecutive moves. Therefore if you consistently move towards your opponent's king every time you get to move, you will at some point be able to get the kings next to each other, giving you a $50\%$ chance of winning on the following turn. 

By Richman's equivalence argument it follows that in bidding chess with continuous money, an advantage in bankroll no matter how small will allow you to win the game with two bare kings! 
Figuring out how to actually win with a bankroll of $1/2+\epsilon$ is a nice little exercise (the number of moves needed will go to infinity as $\epsilon \to 0$). 

Actually that winning strategy can be carried out just as well even if your king is restricted to squares of only one color, say the dark squares. This means that if White has a light-squared bishop (bishop that moves on the light squares) and there are no other pieces except the kings, then as soon as the black king gets to a dark square, the bishop loses its value. It becomes a ghost that can neither attack the black king, nor defend the white one. 

If we play random turn chess from the position in Figure~\ref{F:bishop} (left), then in case White wins the first three turns, they can win by playing Bg2--h3--d7xa4 (or any of a number of other ways to capture the black king in three moves). And this is actually the only use White can have of their bishop. If Black wins any of the first three coin flips, they will move their king to a dark square and the chances will be even. White's winning chances are therefore $1/8$ more than Black's, which means that the value is $9/16$. 

\begin{figure} [h]
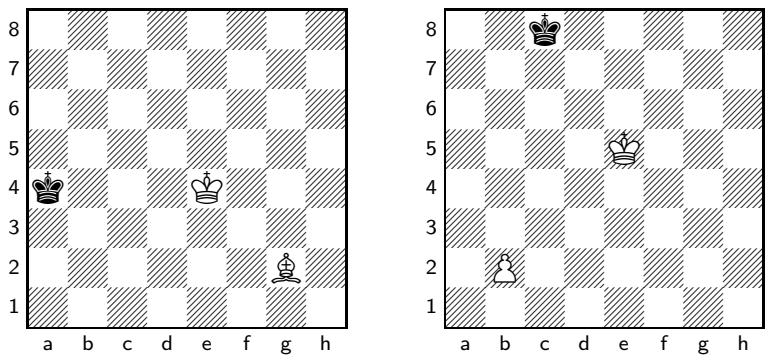

\begin{center}
\storechessboardstyle{8x8}{maxfield=h8, smallboard=<20>, marginbottomwidth=0.5em}
\chessboard[style=8x8, setwhite={Bg2,Ke4},addblack={Ka4},showmover=false]
\storechessboardstyle{8x8}{maxfield=h8, smallboard, marginbottomwidth=0.5em}
\chessboard[style=8x8, setwhite={Pb2,Ke5},addblack={Kc8},showmover=false]
\end{center}
\caption{Left: A bishop endgame worth $9/16$. As soon as Black gets to move, the bishop becomes worthless. Right: A pawn endgame worth  $33/64$. Whenever Black moves their king to the b-file, the pawn is neutralized.}
\label{F:bishop}
\end{figure} 

A similar analysis shows that the position in Figure~\ref{F:bishop} (right) has value $33/64$. As soon as Black gets to move their king to the b-file, the white pawn will be neutralized, since Black then moves down the b-file and captures the pawn unless White chooses to put the kings next to each other before that. Therefore the only use White can have of their pawn comes from the possibility of capturing the black king in five consecutive moves through b2--b4--b5--b6--b7xc8.


Other positions can be evaluated by slightly more sophisticated probabilistic arguments.
\begin{figure} [h]
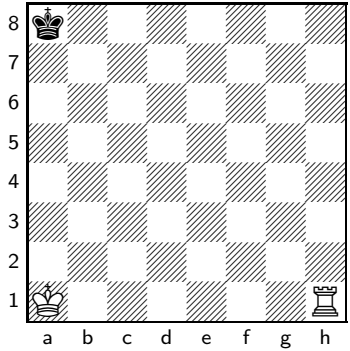

\storechessboardstyle{8x8}{maxfield=h8, smallboard, marginbottomwidth=0.5em}
\begin{center}
\chessboard[style=8x8, setwhite={Rh1,Ka1},addblack={Ka8},showmover=false]
\end{center}
\caption{A rook ending worth $1 -\frac{3^6}{2\cdot 4^6} = \frac{7463}{8192}$.}
\label{F:rookending}
\end{figure} 
In the position in Figure~\ref{F:rookending}, White's best option is to attack the black king from the side with Rh8. Black on the other hand will play Ka7 (or Kb7), trying to get their king as quickly as possible in close combat with the white king. So the black king will move down the a-file going straight for the white king. Meanwhile, White will attack the black king from the side with the rook whenever they can. In order for Black's plan to work, Black therefore has to succeed in playing six moves (from a8 to a2) without White getting two consecutive turns (in which case the rook would capture the black king), and then winning the final coin flip when the kings are face to face. The probability of Black winning at least one of two coin flips is $3/4$, and therefore the probability of Black's king getting to a2 without being captured by the rook is $(3/4)^6$. Black's winning chances are therefore $(1/2)\cdot (3/4)^6 = 729/8192$. Making the analysis rigorous would require dismissing other moves as inferior, but that is relatively straightforward.  

\section{Finite $n$ thresholds and their limits} \label{S:thresholds} 

We have written a computer program (the code, in the \emph{Processing} language, is available on request) that has calculated the values of all positions with three pieces on the board. The program starts by calculating certain thresholds that we now describe.

For each $n$ and each position $P$, we can define a threshold $\alpha_n(P)\in [0,1]$ such that if Black's bankroll is strictly smaller than $\alpha_n(P)$, White can force a capture of the black king in at most $n$ moves, while if Black has strictly more money than $\alpha_n(P)$, the black king can survive for at least $n$ more moves. At the exact threshold, the outcome might depend on the tiebreak rule.

These numbers satisfy the recursion
\begin{equation}\label{computational} \alpha_{n+1}(P) =\frac{\max_w \alpha_n(P_w) + \min_b \alpha_n(P_b)}2,\end{equation}
where $w$ and $b$ range over the white and black move options from position $P$, and $P_w$ and $P_b$ are the positions reached from $P$ by these moves. The ``boundary conditions'' are given by setting $\alpha_n(P) = 0$ or 1 respectively (for all $n$) in positions where the white or black king has already been captured, and $\alpha_0(P)=0$ otherwise. 

Similarly we define thresholds $\beta_n(P)$ as the amount of money that Black needs in order to force a capture of the White king in at most $n$ moves. The $\beta$-thresholds satisfy the same recursive equation \eqref{computational} as $\alpha$, but start from setting $\beta_0(P) = 1$ in positions where both kings remain on the board. Notice that both $\alpha$ and $\beta$ measure the quality of a position from White's perspective in the sense that a higher value is better for White. 

It follows by induction that all these values are dyadic rational numbers, that is, rational numbers with a power of 2 in the denominator. For each position $P$, the sequence $\alpha_n(P)$ is nondecreasing and bounded above by 1. Therefore there is a limit $\alpha(P)$ which is the smallest number such that if Black's bankroll is below $\alpha(P)$, White can force a win. These limits satisfy the same equations:
\begin{equation}\label{a} \alpha(P) =\frac{\max_w \alpha(P_w) + \min_b \alpha(P_b)}2.\end{equation}
Similarly there is a limit $\beta(P)$ of $\beta_n(P)$ which is the amount of money that Black needs in order to force a capture of the white king. It is clear from the definitions that 
\begin{multline} 0\leq \alpha_0(P)\leq \alpha_1(P)\leq \alpha_2(P)\leq  \dots \\ \leq \alpha(P)\leq \beta(P) \leq \\ \dots\leq \beta_2(P) \leq \beta_1(P)\leq \beta_0(P) \leq 1.\end{multline}

In the examples we have discussed, the values $\alpha$ and $\beta$ have been equal. 
But it may also happen (see Section \ref{S:sharp}) that $\alpha(P) < \beta(P)$, so that if Black's bankroll is in the interval between $\alpha(P)$ and $\beta(P)$, the game is drawn in the sense that none of the players can force the capture of the opponent's king. 

For impartial games, such positions (with so-called \emph{nontrivial Richman intervals}) can occur only in games with infinitely many positions, and an example is demonstrated in \cite[Figure 10]{LLPSU}.

The numbers $\alpha$ and $\beta$ can also be interpreted as the probabilities that White can win, and avoid losing, respectively, in the random turn game. Notice that the strategy that achieves the maximal probability $\alpha$ of winning may be different from the one that achieves the maximal probability $\beta$ of not losing.

Our computer program starts by calculating the numbers $\alpha_n(P)$ and $\beta_n(P)$ for all positions with three pieces, and $n$ up to several thousands. This requires working with ``big integers'' since the values are rational numbers with $n$-bit numerators (and denominator $2^n$), but $\alpha_{1000}(P)$ and $\beta_{1000}(P)$ for instance can be computed in a few minutes without any particular optimization. 

\section{Rationality of the limits $\alpha(P)$ and $\beta(P)$}\label{S:rational}

In the examples of Section~\ref{S:examples}, all values were dyadic rational numbers, but this need not always be the case. As we shall see in Section~\ref{S:nondyadic}, there are positions whose values have non-2-power denominators. However, $\alpha(P)$ and $\beta(P)$ are always rational numbers. This holds in general for games with finitely many positions. Suppose therefore that Black and White play a bidding (or random turn) game defined by a finite set of positions, where each position has prescribed sets of white options and black options (other positions to which White and Black can move respectively). Some positions are designated as winning for one of the players. 

Suppose also that $\alpha_n(P)$, $\beta_n(P)$ and their limits $\alpha(P)$ and $\beta(P)$ are defined as in Section~\ref{S:thresholds}.

\begin{Prop}
For every position $P$, $\alpha(P)$ and $\beta(P)$ are rational numbers. 
\end{Prop}

\begin{proof}
Consider the system of equations 

\begin{equation}\label{theoretical} 
x(P) =\frac{\max_w x(P_w) + \min_b x(P_b)}2,
\end{equation}
with the extra constraints that $0\leq x(P)\leq 1$, and that $x(P)=0$ or 1 respectively for positions defined as winning for one of the players (when the white or black king is already captured). In \eqref{theoretical} we have only replaced the symbol $\alpha$ in \eqref{a} by $x$ to indicate that these are variables in a system of equations. A solution to \eqref{theoretical} (including boundary conditions) will be called a \emph{Richman function}, following \cite{LLPU, LLPSU}. 

It follows by induction on $n$ that $\alpha_n$ is a lower bound on any Richman function and similarly $\beta_n$ is an upper bound. Therefore among all Richman functions, $\alpha$ simultaneously minimizes all values, and $\beta$ simultaneously maximizes them.

Since the right hand side of \eqref{theoretical} involves both a minimization and a maximization, the system is inherently nonlinear. But suppose that for each position $P$ we choose (arbitrarily) move options $W$ and $B$ to positions $P_W$ and $P_B$ respectively. Then in order to check whether there is a Richman function for which $x(P_W) =  \max_w x(P_w)$ and $x(P_B) = \min_b x(P_b)$ for every $P$, we can replace \eqref{theoretical} by a linear system of inequalities: For each position $P$ and any white and black options $P_w$ and $P_b$ respectively, we impose the constraints

\begin{equation}
  \left\{ \,
    \begin{IEEEeqnarraybox}[][c]{l?s}
      \IEEEstrut
      x(P) \geq \frac{x(P_w) + x(P_B)}2, \\
      x(P) \leq \frac{x(P_W) + x(P_b)}2, 
      \IEEEstrut
    \end{IEEEeqnarraybox}
\right.
  \label{linear}
\end{equation}
again together with the boundary conditions that $x(P)=1$ if White has won and $x(P)=0$ if Black has won.

Notice that despite the apparent slack in \eqref{linear}, every solution to the system \eqref{linear} is also a solution to the system \eqref{theoretical}: Assuming that \eqref{linear} holds, we have
\begin{multline} \notag
x(P) \leq \frac{x(P_W) + \min_b x(P_b)}2 \leq \frac{\max_w x(P_w) + \min_b x(P_b)}2 \\ \leq \frac{\max_w x(P_w) + x(P_B)}2 \leq x(P).
\end{multline}
Conversely, every Richman function will yield, by choosing $P_W$ and $P_B$ as a minimizing and maximizing option respectively, a solution to the system \eqref{linear}.

The system \eqref{linear} is a set of linear constraints, and whenever the set of solutions is nonempty, it is a polytope with vertices in rational points. Since there are only finitely many ways of choosing $P_W$ and $P_B$, and each resulting system (if solvable) has rational minimum and maximum values for $x(P)$, it follows that $\alpha(P)$ and $\beta(P)$ are rational for every position $P$.
\end{proof}

Similar arguments are given in \cite{LLPU, LLPSU}, but the situation considered in those papers is slightly simpler since for finite impartial games there is only one Richman function. The example in \cite[Figure~10]{LLPSU} shows that for a game with infinitely many positions, the minimum and maximum Richman functions are not necessarily rational. Although the authors of \cite{LLPSU} seem to have overlooked this, in their example, 
$r(k) = (\sqrt{5}-1)^k/2^{k+1}$.
 
\section{Guessing a rational limit}\label{S:guess}
It is obviously not feasible to solve all the linear systems of the form \eqref{linear} in order to find the value of a position. On the other hand our computer program will calculate the finite $n$ thresholds $\alpha_n(P)$ and $\beta_n(P)$ for all positions with two or three pieces and all $n\leq 1000$ (say) in just a few minutes, and this ought to give a good indication of what the limits $\alpha(P)$ and $\beta(P)$ are.
The computation reveals that for all three-piece positions, \[\beta_{1000}(P) - \alpha_{1000}(P) < 10^{-91}.\] This clearly suggests that $\alpha(P) = \beta(P)$ for all three-piece positions (and we describe in Section~\ref{S:verifying} how to verify this). If this is correct, then for sufficiently large $n$, the common value of $\alpha(P)$ and $\beta(P)$ will be the rational number with the smallest denominator in the  interval $[\alpha_n(P), \beta_n(P)]$. We do not know of any simple and useful estimates of how large this $n$ has to be, or of how large the denominators of $\alpha(P)$ and $\beta(P)$ can be (they can be fairly large, see Section~\ref{S:knight}). 

But an obvious thing to do is to let $s_n(P)$ be the simplest rational number (the one with smallest denominator) in the interval $[\alpha_n(P), \beta_n(P)]$, and check whether $s_n$ is a Richman function, in other words whether $x(P) = s_n(P)$ yields a solution to the system \eqref{theoretical}. The numbers $s_n(P)$ can be computed from $\alpha_n(P)$ and $\beta_n(P)$ using a standard technique based on comparing the continued fraction expansions of $\alpha_n(P)$ and $\beta_n(P)$.

We let our computer program calculate $s_n(P)$ for reasonably small $n$ using exact rational arithmetic, and then counted for each $n$ the number of equations in the system \eqref{theoretical} that were violated when putting $x(P) = s_n(P)$. It turns out that at $n=2644$, this number drops to zero, and all equations are satisfied. Actually $s_n$ stabilizes for the bishop endgame already at $n=30$, for the queen endgame at $n=156$, and for the rook endgame at $n=331$. The bulk of the computation is then devoted to the knight endgames and a smaller set of pawn endgames potentially leading to knight promotion. 

%

\section{Verifying that the guesses are correct} \label{S:verifying}
We have now found an explicit Richman function $x(P) = s_{2644}(P)$. This shows that $\alpha(P)\leq s_{2644}(P) \leq \beta(P)$ for every $P$. We will show that equality holds, but at this point it is still conceivable that some of these inequalities are strict. Notice that since $s_{2644}(P)$ will be in the interval $[\alpha_n(P), \beta_n(P)]$ for every $n$, we have $s_n(P) = s_{2644}(P)$ whenever $n\geq 2644$ and we may set $s(P) = s_{2644}(P) = \lim_{n\to\infty}s_n(P)$. 

Although the number of violated equations in \eqref{theoretical} does not consistently decrease as $n$ increases, once it drops to zero so that the system is satisfied, it must remain zero for all larger $n$.

Whenever $x$ is a Richman function, it provides a certificate that White cannot win random turn chess from a position $P$ with probability larger than $x(P)$, and that analogously Black cannot win with probability larger than $1-x(P)$. This is because it provides each player with what we might call an \emph{$x$-greedy} strategy: Each time it is your turn, you choose to move in such a way that you maximize $x$ if you are White, and minimize $x$ if you are Black. 

If White follows an $x$-greedy strategy from a position $P=P_0$, then no matter what strategy Black adopts, the expectation $\mathbf{E}\left[x(P_n)\right]$ of the value of $x$ at the position $P_n$ reached after $n$ moves will satisfy \begin{equation} \label{expineq} \mathbf{E}\left[x(P_n)\right] \geq x(P_0).\end{equation}
Here we use the convention that whenever one of the kings is captured, the resulting terminal position will remain on the board at all subsequent times. We consider the White and Black strategies to be fixed, and the expectation is over the results of the coin flips.

It follows from \eqref{expineq} that the probability of Black having won the game after $n$ moves cannot exceed $1-x(P)$ for any $n$. Similarly, if Black follows an $x$-greedy strategy, White cannot win with probability greater than $x(P)$.

To verify that $\alpha = \beta = x$, we would like to obtain a stronger certificate showing that White can actually win with probability $x(P)$, and that Black can win with probability $1-x(P)$. When the numbers $x(P)$ have been computed explicitly,  this can actually be effectively checked. 


We assume that we have computed a table of all positions (with up to three pieces) and the values of a Richman function $x$ (in our case $x = s_{2644}$). We describe how to verify that $\alpha(P) = x(P)$ for every $P$.

We wish to exhibit a strategy for White in the random turn game which is $x$-greedy and at the same time has the property that regardless of Black's strategy, the game will terminate with probability 1. If there exists such a strategy, then in view of \eqref{expineq}, if we play from a position $P$, White will win with probability at least $x(P)$. Since this is best possible, it follows that $\alpha = x$. 

To find such a strategy we define a sequence of sets of positions as follows: Let $T_0$ be the set of all terminal positions, and 
for $n\geq 0$, let $T_{n+1}$ be the set of positions that either belong to $T_n$, or have an $x$-greedy white move option to a position $T_n$, or have \emph{all} their black move options to positions in $T_n$. These sets are defined by a closure operation and we can therefore effectively compute the sequence of sets until they stabilize, by making a table where each position $P$ is labeled with the smallest $n$ for which $P\in T_n$, if there is such an $n$. 
When for some $n$, $T_{n+1} = T_n$, the process stabilizes and we let $T=T_n$.

The strategy for White now consists in always playing $x$-greedy moves, and whenever there is an $x$-greedy move option to $T$, choosing such a move to a position with minimal label, that is, belonging to $T_i$ for the smallest possible $i$. 

Following this strategy, White will guarantee that the positions in $T$ are \emph{transient} in the sense that they will only be visited a finite number of times. This is because whenever we reach a position in $T_i$ (for $i>0$), either White has \emph{at least one} move to $T_{i-1}$, or \emph{all} Black's moves lead to $T_{i-1}$, and in either case there is (at least) a $50\%$ chance that the next move will lead to a position in $T_{i-1}$. 

Consequently, each time we reach a position in $T_n$, the probability that the game will terminate in another $n$ moves is at least $2^{-n}$. Therefore with probability 1, the game cannot visit such a position infinitely many times. 

\begin{Prop}
We have $\alpha = x$ if and only if all positions $P$ with $x(P)>0$ belong to $T$.
\end{Prop}

\begin{proof}
If all positions where $x$ is nonzero belong to $T$, then no such position can be visited infinitely many times. Consequently the game will either terminate or end up in an infinite sequence of positions where $x$ takes value zero. Since White plays $x$-greedily, if the game starts from a position $P$, in view of \eqref{expineq}, White must win with probability at least $x(P)$, showing that $\alpha(P) = x(P)$. 

Conversely, if there is a position $P$ with $x(P)>0$ which is not in $T$, then every White strategy pretending to win with probability given by $x$ is flawed in one of two ways: Either it consequently plays $x$-greedy moves, in which case it can't win starting from $P$ (since Black can avoid moving into $T$). Or it does \emph{not} always play $x$-greedy moves, in case again it can't always win with probability given by $x$ (provided Black plays $x$-greedily).
\end{proof}

This shows that once we have verified that $s=s_{2644}$ is a Richman function, we can effectively check whether or not $\alpha = s$, and similarly whether or not $\beta = s$. It turns out that in the three-piece endgames, all positions belong to $T$, which shows that $\alpha = s$. Notice however that the definition of $T$ is not symmetrical with respect to the two players, so that in order to verify that $\beta = s$, we would need to check a set $T'$ defined similarly but from Black's perspective.  

In order to verify that all positions belong to $T$ (and similarly to $T'$), we actually only need to investigate a small set of positions. Let us say that a position $P$ is \emph{quiescent} (relative to $x$) if $$\min_b x(P_b) = \max_w x(P_w).$$

\begin{Prop}
If $x$ is a Richman function and all positions that are quiescent with respect to $x$ belong to $T$, then all positions belong to $T$. 
\end{Prop}

\begin{proof}
This follows by induction on the number of positions that have $x$-values larger than $x(P)$ for a given position $P$. Suppose that all quiescent positions belong to $T$, and that also all positions $Q$ with $x(Q)>x(P)$ belong to $T$.

If $P$ is quiescent, then by assumption $P$ belongs to $T$. If $P$ is not quiescent, then either $$\min_b x(P_b) <x(P) < \max_w x(P_w),$$ or $$\max_w x(P_w) < x(P) < \min_b x(P_b).$$
In the former case there is a white option to a position $P_w$, which must belong to $T$ since $x(P_w)>x(P)$. In the latter case all Black options lead to positions $P_b$ which similarly must belong to $T$. In either case, $P$ must belong to $T$.
\end{proof}

After finding the Richman function $x=s_{2644}$, we let our computer program list all quiescent positions relative to this function. They turn out to fall in four categories, three of which were discussed in Section~\ref{S:examples}:

\begin{itemize}
\item Bare kings: Only the two kings on the board. 
\item Ghost bishop: The black king and white bishop on squares of opposite color.
\item Blocked pawn: The black king on the same file as a white pawn, and in front of it. 
\item Cornered king: There are eight positions similar to the one shown in Figure~\ref{F:corneredking}, where the white king is trapped in front of its own pawn near a corner and cannot get out without moving next to the black king. For the king to be trapped on the a-file, the white pawn has to be on a6 or a7. If it is on a6, the black king has to be on c7, and if it is on a7, the black king can be on c7 or c8. And there are four similar positions where the king is trapped on the h-file.  
\end{itemize}

\begin{figure} [h]
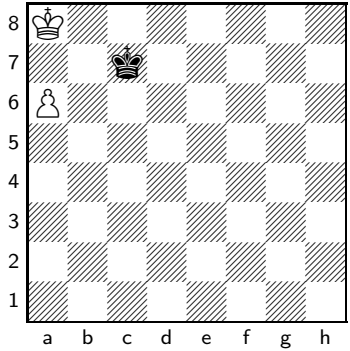

\storechessboardstyle{8x8}{maxfield=h8, smallboard=<20>, marginbottomwidth=0.5em}
\begin{center}
\chessboard[style=8x8, setwhite={Ka8, Pa6},addblack={Kc7},showmover=false]
\end{center}
\caption{A quiescent position: The white king is cornered and cannot get out without challenging the black king.}
\label{F:corneredking}
\end{figure}

In all these positions, it is easy to see that both White and Black have strategies that win with probability $1/2$ by consistently moving the king towards the opponent's king, except in the \emph{blocked pawn} cases where Black should first capture the white pawn. More precisely, the \emph{bare kings} and \emph{ghost bishop} positions belong to $T_7$ and $T'_7$, since a player can capture the opponent's king in at most 7 $s$-greedy moves with a favorable sequence of coin flips. Similarly the \emph{cornered king} positions belong to $T_2$ and $T'_2$, while the \emph{blocked pawn} positions belong to $T_7$ and $T'_{13}$, since being restricted to $s$-greedy moves, it might take Black up to 13 moves to first capture the white pawn and then go after the white king. 

We can therefore conclude that $\alpha(P) = \beta(P) = s_{2644}(P)$ for all three-piece endgames.

\section{Non-dyadic values} \label{S:nondyadic}

One might have expected from the discussion in Section~\ref{S:examples} that all values are dyadic rationals, but this is not the case. The position in Figure~\ref{F:nondyadic} with value $249/320$ is the simplest with a non-2-power denominator (on an 8 by 8 board).

\begin{figure} [h]
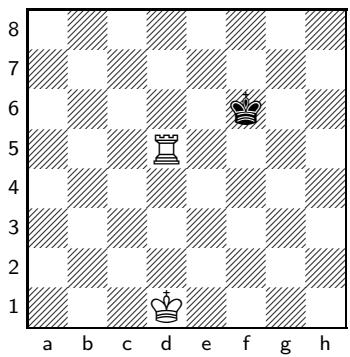

\storechessboardstyle{8x8}{maxfield=h8, smallboard, marginbottomwidth=0.5em}
\begin{center}
\chessboard[style=8x8, setwhite={Kd1, Rd5},addblack={Kf6},showmover=false]
\end{center}
\caption{A position with the non-dyadic value $249/320$.}
\label{F:nondyadic}
\end{figure}

Most rook and queen endgames have 2-power denominators. For queen endgames, the largest denominator is $2^{28}$, but there are also values with denominators divisible by 3, 5, and 17. 
For rook endings, the largest denominator is $229627505902878720 = 2^{40}\cdot 3^3\cdot 5 \cdot 7\cdot 13\cdot 17$, but there are also denominators with a factor 251. 
Bishop endings are relatively simple with only denominators of 1, 2, 4, 8, and 16 occurring. 
     
\section{Knight endgames and huge denominators} \label{S:knight}
 By far the most complex three-piece endgames are the knight endgames (as well as some pawn endgames that lead to knight promotion). The largest denominator in a three-piece endgame occurs for the position in Figure~\ref{F:complex}.
 
%
%
%
%
%
%
%
%
%

\begin{figure} [h]
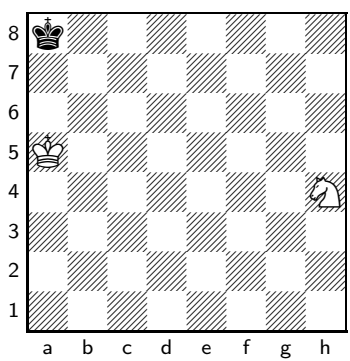

\storechessboardstyle{8x8}{maxfield=h8, smallboard, marginbottomwidth=0.5em}
\begin{center}
\chessboard[style=8x8, setwhite={Nh4,Ka5},addblack={Ka8},showmover=false]
\end{center}
\caption{The most complex three-piece ending.}
\label{F:complex}
\end{figure}
The value of this position is a number with 138-digit numerator and denominator:

\bigskip

118149099210761088839658071450928865980708175943671062283570061

370088990297242487312344048797827448187146592684262495193145202

761460197371
\begin{center} $\Big /$ \end{center}

200453006658428905551436939930457127472327950605425153085344343

480681727125595119114980629492845444447049929082740309543514434

854453248000,

\bigskip

\noindent or approximately $0.5894104617$. The denominator factorizes as \begin{multline}\notag 2^{131}\cdot 3^4 \cdot 5^3 \cdot 7^2\cdot 17\cdot 211\cdot 487 \cdot 63587\cdot 68891\cdot 1894603\\ \cdot 42481581776421430245997\\ \cdot 240980537473228976453730945188262261414394247399.\end {multline}

It is true that this exact number only governs an idealized version of bidding chess with continuous money, but since the number somehow reflects the pattern of optimal moves, the optimal strategies will likely be very intricate also for a reasonable number of chips (although the optimal moves may vary depending on the number of chips \cite{DP}).  

\section{Zugzwang} \label{S:zugzwang}
It was pointed out in \cite{LLPU} that \emph{impartial} bidding games (so-called Richman games) never require negative bids. This does not hold in general for partizan games \cite{DP}. It was speculated in \cite{Beasley} that there might exist positions in bidding chess calling for negative bids, that is, positions where one would prefer the opponent to make the next move. This is indeed the case, and an example is given in Figure~\ref{F:zugzwang}.
\begin{figure} [h]
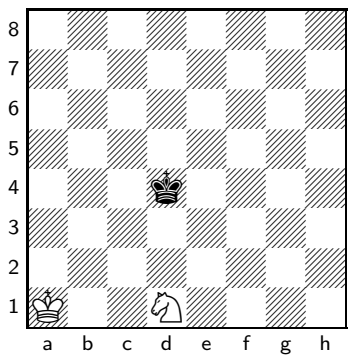

\storechessboardstyle{8x8}{maxfield=h8, smallboard, marginbottomwidth=0.5em}
\begin{center}
\chessboard[style=8x8, setwhite={Ka1, Nd1},addblack={Kd4},showmover=false]
\end{center}
\caption{A position of \emph{zugzwang}: neither player wants to move.}
\label{F:zugzwang}
\end{figure}

This position has value $21073/32256\approx 0.6533$. White's best move is to ``sacrifice'' the knight with Nd1--c3, even though this leads to a position of value only $10489/16128 \approx 0.6504$. The problem is that a move with the king will bring it closer to the black king, while moving the knight will either put it \emph{en prise} (c3 or e3) or move it further from the black king in the knight's metric (b2 and f2 are four knight moves away from d4). Black's best move is Kd4--c4, bringing the value up to $21/32 = 0.65625$.

\section{Pawn promotion}
Since there is no stalemate in bidding chess, we only need to consider promotion to queen or knight. A rook or bishop can never be better than a queen. In some positions there is only a tiny difference in value between promoting to knight and promoting to a queen. For instance, in the position in Figure~\ref{F:promotion} (left), White to move should play d8N!, obtaining a position of value $205/256\approx 0.80078$, while a promotion to queen gives a value of only $3279/4096\approx 0.80054$. However, if we move the entire position one step to the right as in Figure~\ref{F:promotion} (right), the knight-promotion still leads to a position of value $205/256$, while e8Q! gives the slightly higher value of $3285/4096\approx0.80200$! 

\begin{figure}
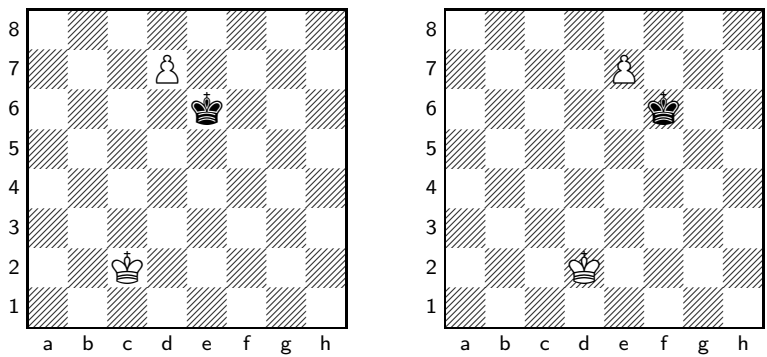

\begin{center}
\storechessboardstyle{8x8}{maxfield=h8, smallboard}
\chessboard[style=8x8, setwhite={Kc2, Pd7},addblack={Ke6},showmover=false]
\storechessboardstyle{8x8}{maxfield=h8, smallboard}
\chessboard[style=8x8, setwhite={Kd2, Pe7},addblack={Kf6},showmover=false]
\end{center}
\caption{Left: White's best move is to promote to a knight. Right: White should promote to a queen.}
\label{F:promotion}
\end{figure}

\section{Positions where $\alpha<\beta$} \label{S:sharp}
We have shown that for all three-piece endgames, $\alpha(P)=\beta(P)$, but we have no ``theoretical'' explanation for why this must be so. There are conditions under which bidding games must be \emph{sharp} in this sense \cite{LLPU, LLPSU}, but such conditions do not seem to be met in chess. And if we allow  more pieces on the board, it is easy to construct positions that have so-called \emph{nontrivial Richman intervals}, that is, where $\alpha(P) < \beta(P)$. 

\begin{figure} [h]
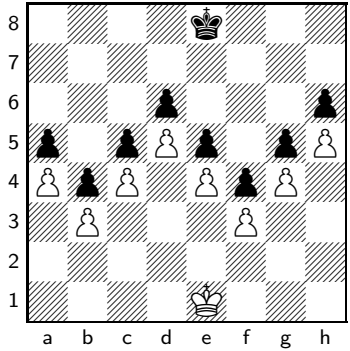

\storechessboardstyle{8x8}{maxfield=h8, smallboard}
\begin{center}
\chessboard[style=8x8, setwhite={Pa4,Pb3, Pc4, Pd5, Pe4, Pf3, Pg4, Ph5,Ke1},addblack={Ke8, Pa5, Pb4, Pc5, Pd6, Pe5, Pf4, Pg5, Ph6},showmover=false]
\end{center}
\caption{A position where $\alpha < \beta$.}
\label{F:draw}
\end{figure}

An example is given in Figure~\ref{F:draw}, where we claim that $\alpha\leq 1/4$ and $\beta\geq 3/4$. In other words, a player with more than $1/4$ of the money need not lose. For instance, if White tries to break through the wall of pawns by playing the king to d4 and capturing at e5, Black will go \emph{all in} when the white king has reached d4. Black will then have more money than White after the capture on e5, and will be able to recapture with the d6-pawn. 

\section{Other board sizes}
Mathematically there is of course nothing special about the $8\times 8$ board size, and we have investigated other board sizes as well. The results are similar to those of the 8 by 8 board. 
In particular there are no three-piece endgames with $\alpha<\beta$ for any board size smaller than 8 by 8. 

On the 3 by 4 board, there are quiescent positions of value different than $1/2$. In the position of Figure~\ref{F:quiescent}, White cannot improve their position by any move. 

\begin{figure} [h]
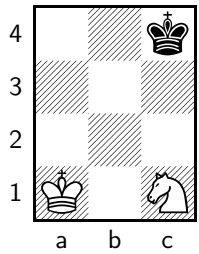

\storechessboardstyle{3x4}{maxfield=c4}
\begin{center}
\chessboard[style=3x4, setwhite={Ka1,Nc1},addblack={Kc4},showmover=false]
\end{center}
\caption{A quiescent position of value $5/8$.}
\label{F:quiescent}
\end{figure}

One peculiarity that occurs on a 3 by 8 board is a value with odd denominator. The following position has the value $653/819$, the denominator factorizing as $3^2\cdot 7\cdot 13$.
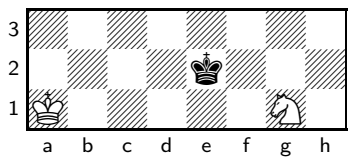
\storechessboardstyle{8x3}{maxfield=h3, smallboard}
\begin{center}
\chessboard[style=8x3, setwhite={Ka1, Ng1},addblack={Ke2},showmover=false]
\end{center}


A curiosity that occurs on a 4 by 4 board is the following position where White will win the random turn game with probability $31/48$, but where the game (provided it is played optimally) will end in a draw if White wins all the coin flips! 

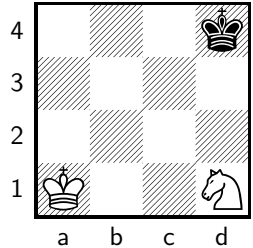
\storechessboardstyle{4x4}{maxfield=d4}
\begin{center}
\chessboard[style=4x4, setwhite={Ka1, Nd1},addblack={Kd4},showmover=false]
\end{center}

Just like the similar position on the 8 by 8 board, this is a \emph{zugzwang}, where White would prefer Black to make the next move. As long as the black king stays in the corner, White's problem is that they can't bring their knight to a square where it threatens the black king without first putting it \emph{en prise}. If White has to move, there are three optimal moves, Ka2, Kb1, and Nb2, all three decreasing the value from White's perspective to $61/96$. If White then gets to move again, their best option is to move back to the diagram position (or to the equivalent position with the knight on a4). So as long as White ``wins'' all the coin flips, they will move back and forth, waiting for Black to have to move their king. 

Whenever a position has a non-dyadic value, there must be some infinite sequence of coin flips that causes the random turn game to go on forever under optimal play. What is a bit unusual here is that that sequence is one where the same player wins them all.

\end{document}